\documentclass[12pt]{csdm}

\usepackage{url,floatflt,helvet,times}
\usepackage{psfig,graphics}
\usepackage{mathptmx,amssymb,amsmath,bm}
\usepackage{float,tcolorbox,xcolor}
\usepackage[bf,hypcap]{caption}
\usepackage{graphicx,wrapfig,subfig}
\usepackage{setspace,enumerate,bookmark,placeins}
\usepackage{multirow,multicol}
\usepackage[curve,frame,line,arrow,matrix]{xy}
\usepackage[latin5]{inputenc}
\usepackage{tikz,calc}
\usetikzlibrary{arrows,backgrounds,automata}
\usetikzlibrary{decorations.markings}
\tikzstyle{vertex}=[circle, draw, inner sep=2pt, minimum size=6pt]

\emergencystretch=\maxdimen
\allowdisplaybreaks

\def\noi{\noindent}

\def\firstpage{13}
\def\thevol{2}
\def\thenumber{1}
\def\theyear{2018}

\begin{document}
{\fontfamily{cmss}\selectfont

\titlefigurecaption{\hspace{0.3cm}{\hspace{0.6cm}\LARGE \bf \sc \sffamily\color{white} Contemporary Studies in Discrete Mathematics}}

\title{\sc \sffamily On Chromatic Core Subgraph of Simple Graphs}

\author{\sc\sffamily Johan Kok$^{1,\ast}$, Sudev Naduvath$^2$, Jerlin Seles$^3$ and U Mary$^4$}
\institute {$^{1,2}$Centre for Studies in Discrete Mathematics, Vidya Academy of Science \& Technology, Thrissur, India. \\
$^{3,4}$Department of Mathematics, Nirmala College for Women, Coimbatore, India.}


\titlerunning{ On Chromatic Core Subgraph of Simple Graphs}
\authorrunning{J.Kok, S.Naduvath, J. Seles \& U. Mary}

\mail{kokkiek2@tshwane.gov.za}

\received{5 July 2017}
\revised{18 December 2017}
\accepted{21 January 2018}
\published{23 February 2018.}

\abstracttext{ If distinct colours represent distinct technology types that are placed at the vertices of a simple graph in accordance to a minimum proper colouring, a disaster recovery strategy could rely on an answer to the question: ``What is the maximum destruction, if any, the graph (a network) can undergo while ensuring that at least one of each technology type remain, in accordance to a minimum proper colouring of the remaining induced subgraph." In this paper, we introduce the notion of a chromatic core subgraph $H$ of a given simple graph $G$ in answer to the stated problem. Since for any subgraph $H$ of $G$ it holds that $\chi(H) \leq \chi(G)$, the problem is well defined. }
\keywords{Chromatic colouring, chromatic core subgraph, structural size, structor index.}

\msc{05C15, 05C38, 05C75, 05C85.}
\maketitle

\section{Introduction}

For general notation and concepts in graphs and digraphs see \cite{3,6,10}. Unless mentioned otherwise all graphs $G$ are finite, undirected simple graphs. The number of vertices of a  graph $G$ is called its \textit{order} and is denoted by $\nu(G) = n \geq 1$ and the number of edges of edges of $G$ is called its \textit{size}, denoted by $\varepsilon(G)= p \geq 0$. The minimum and maximum degree of $G$ are denoted by $\delta(G)$ and $\Delta(G)$ respectively. The degree of a vertex $v \in V(G)$ is denoted by $d_G(v)$ or when the context is clear, simply as $d(v)$. Also, the vertices and the edges of a graph are together called the \textit{elements} of a graph.

Recall that a \textit{proper vertex colouring} of a graph $G$ denoted $\varphi:V(G) \mapsto \mathcal{C}=\{c_1,c_2,c_3,\dots,c_\ell\}$, a set of distinct colours, is a vertex colouring such that no two distinct adjacent vertices have the same colour. The minimum number of colours required in a proper vertex colouring of a graph $G$ is called the \textit{chromatic number} of $G$ and is denoted by $\chi(G)$. When a vertex colouring is considered with colours of minimum subscripts the colouring is called a \textit{minimum parameter colouring}. Unless stated otherwise we consider minimum parameter colour sets throughout this paper.

\section{Chromatic Core Subgraph}

If distinct colours represent distinct technology types that are placed at the vertices of a simple graph in accordance to a minimum proper colouring, a disaster recovery strategy could rely on an answer to the question about the maximum destruction the graph (a network) can undergo while ensuring that at least one of each technology type remain in accordance to a minimum proper colouring of the remaining induced subgraph. In view of this question, we introduce the notion of a chromatic core subgraph $H$ of a given simple graph $G$ in answer to the stated problem.

Observe that the notion of a chromatic core subgraph $H$ of a given simple graph $G$ resembles the technique used in \cite{10}, (see Theorem 8.6.19 of \cite{10}) in which, for $\chi(G)=k$, the successive deletion of vertices from graph $G$ are utilised to obtain a subgraph $H$ such that $\chi(H-v) = k-1$. The essential difference is the fact that a chromatic core graph is an induced subgraph with minimum \textit{structor index} $si(G) = \nu(G) + \varepsilon(G)$.

The \textit{structural size} of a graph $G$ is measured in terms of \textit{structor index}. If $si(G)=si(H)$ the graphs are of equal structural size but not necessarily isomorphic. We say that a graph among $G$ and $H$ is the smaller when compared to the other if its structor index is smaller than that of the latter. 

\begin{definition}\label{Defn-2.1}{\rm 
A \textit{chromatic core subgraph} of a finite, undirected and simple graph $G$ of order $\nu(G)=n\geq 1$ is a smallest induced subgraph, say $H$, (smallest in terms of $si(H)$) such that $\chi(H) = \chi(G)$.
}\end{definition}

Note that, up to isomorphism, a chromatic core subgraph is not necessarily unique. Definition \ref{Defn-2.1} essentially requires only the conditions $\min(\nu(H)+\varepsilon(H))$ and that $\chi(H) = \chi(G)$. It is easy to see that a graph $G$ with $\chi(G)=1$ can only be a null graph (edgeless graph). Hence, $K_1$ is a chromatic core subgraph of any graph $G$ with $\chi(G)=1$. Also, for all graphs $G$ with $\chi(G)=2$ (bipartite graphs), any edge $uv \in E(G)$ is a chromatic core subgraph of $G$. Therefore, the bipartite graphs have $\varepsilon$ chromatic core subgraphs.

From a comp\^{o}nent\u{a} analysis perspective (see \cite{8}), a finite, simple graph may have a number of connected components say, $H_1, H_2, H_3,\dots, H_k$ such that $\bigcup\limits_{i=1}^{k}V(H_i) = V(G)$ and $\bigcup\limits_{i=1}^{k}E(H_i) = E(G)$. Renumber the components in decreasing order of the chromatic number of the components. Assume that $\chi(H_1) = \chi(H_2) =\cdots = \chi(H_t) > \chi(H_{t+1})\geq \chi(H_{t+2})\geq \cdots \geq \chi(H_{k})$. Definition \ref{Defn-2.1} implies that if $H'_1, H'_2, H'_3,\dots, H'_t$ are chromatic core subgraphs of $H_1, H_2, H_3,\ldots, H_t$ respectively, then a smallest $\{H'_i:1\leq i \leq t\}$ is a chromatic core subgraph of $G$. It means that we can sensibly investigate only finite, undirected connected simple graphs henceforth. Then, the first important existence lemma can be presented as given below.

\begin{lemma}\label{Lem-2.1}
Any finite, undirected connected simple graph $G$ has a proper subgraph $H$ such that $\chi(H) = \chi(G)$ if and only if $G$ is neither a complete graph nor an odd cycle.
\end{lemma}
\begin{proof}~
From Brook's Theorem it follows that $\chi(G) = \Delta(G)+1$ if and only if $G$ is either an odd cycle or a complete graph. It implies that for only odd cycles and complete graphs, all its proper subgraphs $H$ be such that $\chi(H) < \chi(G)$. All other graphs $G$ will have at least one proper subgraph $H$ with, $\chi(H) = \chi(G)$.
\end{proof}

Clearly, the subgraph $H$ that exists by Lemma \ref{Lem-2.1}, can also be the corresponding induced subgraph, $\langle V(H)\rangle$. This remark is important to link further results with the notion of a chromatic core subgraph.

\subsection{Chromatic core subgraphs of certain classical graphs}

The subsection begins with results for some classical graphs and certain trivial observations. Note that unless mentioned otherwise, only finite, undirected connected simple graphs will be considered.

\begin{proposition}\label{Prop-2.2}
\begin{enumerate}\itemsep0mm
\item[(i)]~ An acyclic graph $G$ has $P_2$ as a chromatic core subgraph.
\item[(ii)]~ An even cycle $C_n,\ n \geq 4$ has $P_2$ as a chromatic core subgraph.
\item[(iii)]~ An odd cycle $C_n,\ n \geq 3$ has $C_n$ as its unique chromatic core subgraph.
\item[(iv)]~ A complete graph $K_n,\ n \geq 1$ has $K_n$ as its unique chromatic core subgraph.
\item[(v)]~ An even wheel $W_n,\ n \geq 4$ has $C_3$ as a chromatic core subgraph.
\item[(vi)]~ An odd wheel $W_n,\ n \geq 3$ has $W_n$ as its unique chromatic core subgraph.
\item[(vii)]~ An even helm graph $H_n,\ n \geq 4$ has $C_3$ as a chromatic core subgraph.
\item[(viii)]~ An odd helm graph $H_n,\ n \geq 3$ has odd $W_n$ as its unique chromatic core subgraph.
\item[(ix)]~ If $diam(G) = 1$ then $G$ is its unique chromatic core subgraph.
\item[(x)]~ If $\chi(G) = \Delta(G) + 1$ then $G$ is its unique chromatic core subgraph.
\item[(xi)]~ If $H$ is a chromatic core subgraph of $G$ then the Mycielski graph $\mu(H)$ is a chromatic core subgraph of the Mycielski graph $\mu(G)$.
\end{enumerate}
\end{proposition}
\begin{proof}~
The results from $(i)$ to $(viii)$ are straight forward and needs no explanation.
\begin{enumerate}\itemsep0mm  
\item[(ix)]~ If $diam(G) =1$, then $G$ is a complete graph and hence the result follows from Part (iv).
\item[(x)]~ By Brook's Theorem, $\chi(G) \leq \Delta(G)$ if and only if $G$ is neither a complete graph nor an odd cycle. Hence, if $\chi(G) = \Delta(G) + 1$, then $G$ is either complete or an odd cycle. So the result follows from Parts (iii) and (iv).
\item[(xi)]~ Since $\chi(\mu(G)) = \chi(G) + 1$, and a chromatic core subgraph $H$ of $G$ is an induced subgraph of $G$, it follows that $\mu(H)$ is the smallest induced subgraph of $\mu(G)$ such that $\chi(\mu(H)) = \chi(H) + 1 = \chi(G) +1 = \chi(\mu(G))$. Therefore, the result.
\end{enumerate}
\end{proof}

The next result is the major result within the context of chromatic core subgraphs.

\begin{theorem}\label{Thm-2.3}
For a given integer $m \in \mathbb{N}$, the smallest graph $G$ with $\chi(G=m$ is the complete graph $K_m$.
\end{theorem}
\begin{proof}~
For $m = 1$, the graph $G = K_1$ is clearly the smallest graph with $\chi(G)=1$ and $si(G)=1$. For $m=2$, the graph $G=K_2=P_2$ is clearly the smallest graph with $\chi(G)=2$ and $si(G)=3$. Also, by the same reasoning, $G=K_3$ is the smallest graph for which $\chi(G)=3$ and $si(G)=6$. Assume that for $4 \leq m \leq \ell$, the smallest graph $G$ for which $\chi(G)=m$ is $G=K_m$ with $si(G)=m+\frac{1}{2}m(m-1)$. Consider $m=\ell$ and let $V(K_\ell)=\{v_1,v_2,v_3,\ldots,v_\ell\}$.
	
Consider $m=\ell+1$. Begin with $K_\ell$ and add an isolated vertex $u$. Clearly, $\chi(K_\ell \cup \{u\})=\ell$ with $si(K_\ell\cup \{u\})=si(K_\ell)+1$. Clearly, adding another isolated vertex or more isolated vertices will not assist to increase the chromatic number. Adding a pendant vertex, adds a count of $2$ elements and hence that is also not optimal. Hence, without loss of generality, add the edge $uv_1$ to obtain $K'_\ell$. Clearly, edges must be added iteratively to obtain $K^{'''\ldots '~(\ell ~times)}_\ell$ to have $\chi(K^{'''\dots '~(\ell ~times)}_\ell) = \ell+1$. Clearly, $1 + \ell < t_{\geq 2} + \ell$ and $K^{'''\dots '~(\ell ~times)}_\ell = K_{\ell + 1}$ with $si(K_{\ell + 1})=(\ell + 1) + \frac{1}{2}(\ell + 1)\ell$. Therefore, the result holds for $m = \ell + 1$ and hence  the results holds for all $m \in \mathbb{N}$, by mathematical induction.
\end{proof}

Recall that the clique number $\omega(G)$ is the order of the largest clique in $G$. Also recall that a perfect graph $G$ is a graph for which any induced subgraph $H$ has, $\omega(H) = \chi(H)$. For a weakly perfect graph, it is only known that $\omega(G) = \chi(G)$, with or without the condition being valid for all the induced proper subgraphs. Therefore, all perfect graphs are weakly perfect as well (see \cite{1}). 

\noi The next corollary is a direct consequence of the aforesaid and of Theorem \ref{Thm-2.3}.

\begin{corollary}
\begin{enumerate}\itemsep0mm
\item[(i)]~ For a weakly perfect graph $G$ its largest clique $H$ is a chromatic core subgraph.
\item[(ii)]~ For a non-weakly perfect graph $G$ the only unique chromatic core subgraph is $G$ itself.
\end{enumerate}
\end{corollary}

Many graph classes are indeed perfect graphs such as: complete graphs, even cycles, paths, bipartite graphs, interval graphs, permutation graphs, split graphs, threshold graphs and others (refer to \cite{1}). A recently introduced class of graph called the linear Jaco graphs (see \cite{7}), will be proven to be perfect as well. 

\begin{definition}\label{Defn-2.2}{\rm 
Let $f(x) = mx + c; x \in \mathbb{N},$ $m,c\in \mathbb{N}_0$. The family of \textit{infinite linear Jaco graphs} denoted by $\{J_\infty(f(x)):f(x) = mx + c; x,m \in \mathbb{N}$ and $c \in \mathbb{N}_0\}$ is defined by $V(J_\infty(f(x))) = \{v_i: i \in \mathbb{N}\},\ A(J_\infty(f(x))) \subseteq \{(v_i, v_j): i, j \in \mathbb{N}, i< j\}$ and $(v_i,v_ j) \in A(J_\infty(f(x)))$ if and only if $(f(i) + i) - d^-(v_i) \geq j$.
}\end{definition}

\noi Figure \ref{fig:FIG2} depicts the Jaco graph $J_{10}(x)$.

\begin{figure}[h]
\centering
\includegraphics[width=0.4\linewidth]{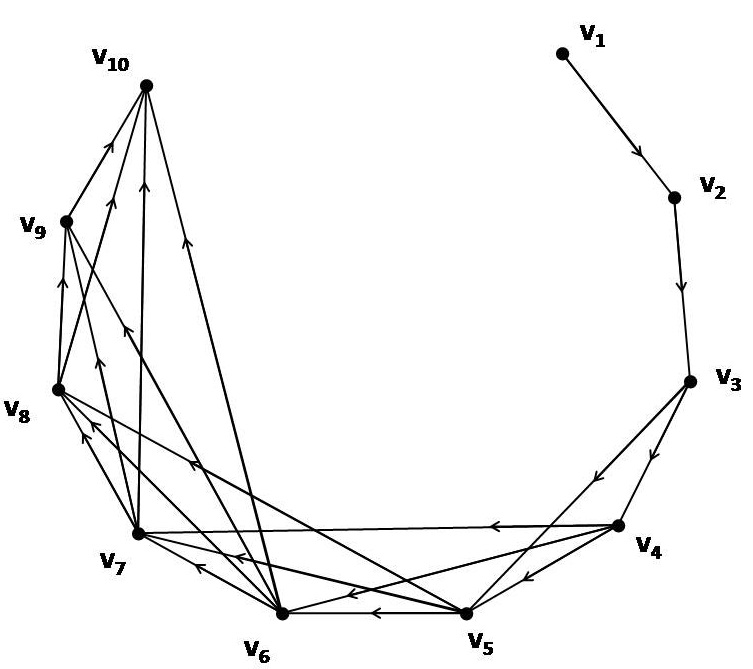}
\caption{Jaco graph $J_{10}(x)$.}
\label{fig:FIG2}
\end{figure}

For this application we consider the underlying finite linear Jaco graph denoted by, $J_n^*(f(x))$. This finite graph is obtained by lobbing off all vertices $v_i,\ i \geq n+1$ together with the corresponding edges.
\begin{theorem}
The family of finite, undirected linear Jaco graphs are perfect graphs.
\end{theorem}
\begin{proof}~
\textit{Case-1:} Let $m \geq 1)$. A graph of order $n \geq 1$ has $2^n-1$ induced subgraphs. Consider the case $f(x)=x$. Note that $J^*_4(x)=P_4$ and hence it and all its induced subgraphs are perfect graphs. Consider $J^*_5(x)$. Here, we only need to verify pefectness of the induced subgraphs $\langle v_1,v_5\rangle$, $\langle v_2,v_5\rangle$, $\langle v_3,v_5\rangle, \langle v_4,v_5\rangle, \langle v_1,v_2,v_5\rangle,\langle v_1,v_3,v_5\rangle$, $\langle v_1,v_4,v_5\rangle$, $\langle v_2,v_3,v_5\rangle$, $\langle v_2,v_4,v_5\rangle$, $\langle v_3,v_4,v_5\rangle$, $\langle v_1, v_2, v_3, v_5\rangle$, $\langle v_1, v_2,v_4,v_5\rangle$, $\langle v_1,v_3,v_4,v_5\rangle$, $\langle v_2,v_3,v_4,v_5\rangle$, $\langle v_1,v_2,v_3,v_4,v_5\rangle$ and $\langle v_5\rangle$.
	
Knowing that the disjoint union of two perfect graphs remains perfect and using Definition \ref{Defn-2.2} makes it easy to verify that each of the induced subgraphs is indeed perfect. Since the result holds for $n = 1,2,3,4,5$, assume that it holds for $6\leq n\leq k$.
	
Consider the case $n=k+1$. It follows, by induction, that the $2^{k+1}-2^k$ additional induced subgraphs are perfect. Hence, the result holds forall $n \in \mathbb{N}$.
	
Finally, since $f(x)=mx+c; x,m \in \mathbb{N},\ c\in \mathbb{N}_0$ are satisfied by only the set of positive integer ordered pairs, (that is, $\{(x,mx + c):m,x \in \mathbb{N}\ \mbox{and}\ c \in \mathbb{N}_0\}$), by similar reasoning to that in the case of $f(x)=x$, it follows by induction that the result holds for all finite, undirected linear Jaco graphs in this subfamily.
	
\textit{Case-2:} Let $m=0$. It follows from Definition \ref{Defn-2.2} that these linear Jaco graphs are all the disjoint union of complete graphs, $K_t,\ t\leq c+1$. Hence, the result holds for this subfamily as well.

Therefore, the result holds for all finite, undirected linear Jaco graphs.
\end{proof}

Figure \ref{fig:FIG4} depicts $J_{15}(3)$ and hence $c=3$ to clarify Case-2 of the above result.

\begin{figure}[h]
\centering
\includegraphics[width=0.4\linewidth]{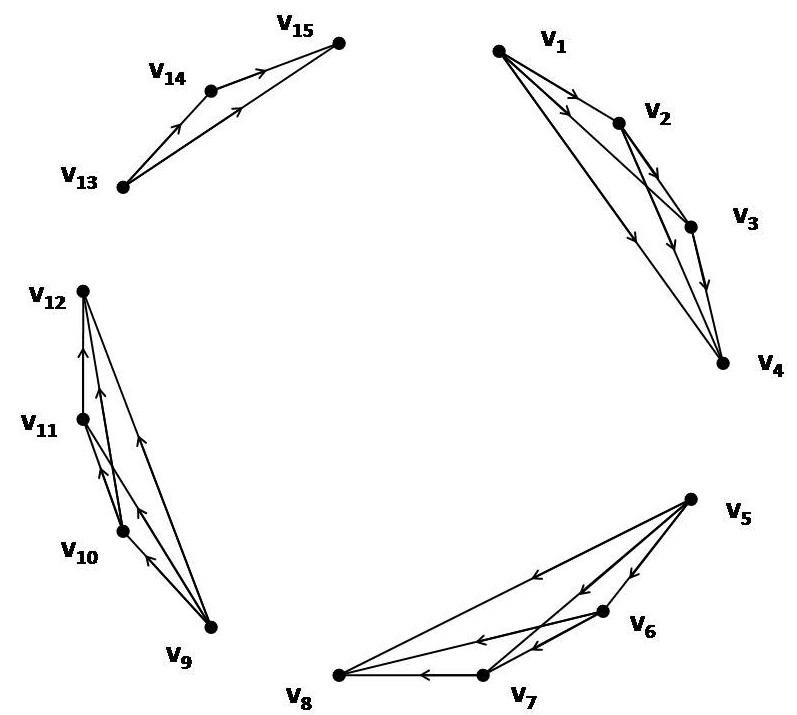}
\caption{Jaco graph $J_{15}(3)$.}\label{fig:FIG4}
\end{figure}

\begin{corollary}
For finite, undirected linear Jaco graphs, the unique chromatic core subgraphs is the respective Hope subgraphs (refer to \cite{7} for the definition of Hope subgraphs Jaco graphs).
\end{corollary}

\section{Chromatic Core Subgraphs of Certain Graph Operations}

In this section, the join, the corona, the Cartesian product, the tensor product, strong product and the lexicographical product (refer to \cite{5} for the definitions of graph products) of graphs $G$ and $H$ are considered.

\subsection{Chromatic Core Subgraphs of Join of Graphs}

Let $G_1(V_1,E_1)$ and $G_2(V_2,E_2)$ be two graphs. Then, their {\em join}, denoted by $G_1+G_2$, is the graph whose vertex set is $V_1\cup V_2$ and edge set is $E_1\cup E_2\cup E_{ij}$, where $E_{ij}=\{u_iv_j:u_i\in G_1,v_j\in G_2\}$ (see \cite{6}).

\begin{theorem}\label{Thm-3.1}
If $G'$ and $H'$ are chromatic core subgraphs of the graphs $G$ and $H$ respectively, then $G'+H'$ is a chromatic core subgraph of $G+H$.
\end{theorem}
\begin{proof}~
Note that for the join $G+H$, we have $\chi(G+H)=\chi(G)+\chi(H)$. If an induced subgraph $G''$ which is smaller than $G'$ exists such that $\chi(G''+ H') = \chi(G + H)$, then $\chi(G'')=\chi(G)$ which is a contradiction. Similar reasoning applies to the cases $G'+H''$ and $G''+H''$.
	
Finally, let $E^*(G+H)$ be the additional edges resulting from the $+operation$ hence, $E^*(G+H) = E(G+H) - (E(G) \cup E(H))$. Since any edge induced subgraph resulting from an edge set $S\subseteq E^*(G+H)$ is a bipartite graph, such subgraph will only suffice if both, $\chi(G) = \chi(H) = 1$. The aforesaid implies that $G = H = K_1$. Hence the result.
\end{proof}

\subsection{Chromatic Core Subgraphs of Corona of Graphs}

The \textit{corona} of two graphs $G_1$ and $G_2$ (see \cite{6}), denoted by $G_1\odot G_2$, is the graph obtained by taking $|V(G_1)|$ copies of the graph $G_2$ and adding edges between each vertex of $G_1$ to every vertex of one (corresponding) copy of $G_2$. The following theorem establishes a range for the curling number of the corona of two graphs. The following theorem discusses the chromatic core subgraphs of corona of graphs.

\begin{theorem}\label{Thm-3.2}
If $G'$ and $H'$ are chromatic core subgraphs of the graphs $G$ and $H$ respectively, then $K_1\circ H'$ is a chromatic core subgraph of $G\circ H$.
\end{theorem}
\begin{proof}~
Since $\chi(G\circ H) = \chi(H) + 1$ and for any $v\in V(G),\ v + H$ is an induced subgraph of $G\circ H$ the result follows by similar reasoning to that in the proof of Theorem \ref{Thm-3.1}. 
\end{proof}

Let $V(G) =\{v_i:1\leq i\leq n\}$ and $V(H) = \{u_j:1\leq j \leq m\}$. For the graph products under discussion the \textit{product vertices} are the set Cartesian product $V(G) \times V(H) = \{(v_i,u_j): 1\leq i \leq n,~1\leq j \leq m \}$. 

\subsection{Chromatic Core Subgraphs of the Cartesian Product of Graphs}

Let $G_1(V_1,E_1)$ and $G_2(V_2,E_2)$ be the two given graphs. The {\em Cartesian product} of $G_1$ and $G_2$ (see \cite{5,6}), denoted by $G_1\Box G_2$, is the graph with vertex set $V_1\times V_2$, such that two points $u=(v_i, u_j)$ and $v=(v_k,u_l)$ in $V_1\times V_2$ are adjacent in $G_1\Box G_2$ whenever [$v_i=v_k$ and $u_j$ is adjacent to $u_l$] or [$u_j=u_l$ and $v_i$ is adjacent to $v_k$].

The following theorem discusses the chromatic core subgraphs of the Cartesian product of graphs.

\begin{theorem}\label{Thm-3.3}
If $G'$ and $H'$ are the chromatic core subgraphs of the graphs $G$ and $H$ respectively, the induced subgraph $or\{G', H'\}$, corresponding to $\max\{\chi(G),\chi(H)\}$ is a chromatic core subgraph of $G\Box H$.
\end{theorem}
\begin{proof}~
Assume that $G$ is of order $n \geq 1$ and $H$ is of order $m \geq 1$. From the definition of the Cartesian product $G\Box H$ of $G$ and $H$, it follows that $E(G\Box H)$ can be partition into edge subsets say, $X$ and $Y$, such that say $X$ can be partitioned into $m$ subsets $X_i,\ 1\leq i \leq m$ and each $\langle X_i\rangle$ is isomorphic to $G$. Similarly, $Y$ can be partitioned in $n$ subsets $Y_j,\ 1\leq j \leq n$ such that each $\langle Y_j\rangle$ is isomorphic to $H$. Because $\chi (G\Box H) = \max\{\chi(G), \chi(H)\}$, the result follows by similar reasoning to that in the proof of Theorem \ref{Thm-3.1}.
\end{proof}

\subsection{Chromatic Core Subgraphs of the Tensor Product of Graphs}

The \textit{tensor product}  of graphs $G_1$ and $G_2$ is the graph $G_1\times G_2$  with the vertex set $V(G_1\times G_2)=V(G_1)\times V(G_2)$ and the vertices $(v,u)$ and $(v',u')$ are adjacent in $G_1\times G_2$ if and only if $vv'\in E(G_1)$ and $uu' \in E(G_2)$. For a vertex $(v_i,u_j)\in G=G_1\times G_2$, we have $d_G(v_i,u_j)=d_{G_1}(v_i)d_{G_2}(u_j)$, where $1\le i\le |V(G_1)|, 1\le j\le |V(G_2)|$ (see \cite{5}). This graph product has various names in the literature such as direct product, categorical product, cardinal product, relational product, Kronecker product, weak direct product or the conjunction of graphs, there is no conclusive result. That $\chi(G\times H) \leq min\{\chi(G), \chi(H)\}$ follows easily. However, the Hedetniemi conjecture (see \cite{6a}) states that $\chi(G \times H)=\min\{\chi(G),\chi(H)\}$. After verifying numerous examples, we are confident to state Hedetniemi's $2$-nd Conjecture as follows.

\begin{conjecture}\label{Conj-3.4}[Hedetniemi's $2$-nd Conjecture]\footnote{In honour of Professor Emeritus Stephen T. Hedetniemi who in 1966 stated the Hedetniemi conjecture}{\rm 
If $G'$ and $H'$ are chromatic core subgraphs of the graphs $G$ and $H$ respectively, then the induced subgraph $or\{G', H'\}$ corresponding to $\min\{\chi(G),\chi(H)\}$ is a chromatic core subgraph of $G\times H$. Moreover, if $\chi(G)=\chi(H)$, then $smallest\{G',H'\}$ is a chromatic core subgraph of $G\times H$.	
}\end{conjecture}

\subsection{Chromatic Core Subgraphs of the Strong Product of Graphs}
 
The {\em strong product} of two graphs $G_1$ and $G_2$ is the graph, denoted by $G_1\boxtimes G_2$, whose vertex set is $V(G_1) \times V(G_2)$, the vertices $(v,u)$ and $(v',u')$ are adjacent in $G_1\boxtimes G_2$ if $[vv'\in E(G_1)~\text{and}~ u=u']$ or $[v=v' ~ \text{and}~uu'\in E(G_2)]$ or $[vv'\in E(G_1)$ and $uu'\in E(G_2)]$ (see \cite{5}). 

\noi Invoking Defintion mentioned above, we claim that the following conjecture is true.

\begin{conjecture}\label{Conj-3.5}
If $G'$ and $H'$ are chromatic core subgraphs of the graphs $G$ and $H$ respectively, then $G'\boxtimes H'$ is a chromatic core subgraph of $G\boxtimes H$.
\end{conjecture}

It is known that $max\{\chi(G),\chi(H)\} \leq \chi(G\boxtimes H) \leq \chi(G)\cdot \chi(H)$. Furthermore, the following facts provide us convincing background to claim that the conjecture that $G'\boxtimes H'$ is a chromatic core subgraph of $G\boxtimes H$

\begin{enumerate}\itemsep0mm
\item[(a)]~ The upperbound in the above inequality is strict in most cases;
\item[(b)]~ In general, $\chi(G+uv) \geq \chi(G)$;
\item[(c)]~ $\chi(G\times H) \leq min\{\chi(G), \chi(H)\}$;
\item[(d)]~ Theorem \ref{Thm-3.3};
\item[(e)]~ $E(G\boxtimes H) = E(G\times H)\cup E(G\square H)$,
\end{enumerate}

\begin{remark}{\rm 
For the bounds $\max\{\chi(G),\chi(H)\} \leq \chi(G\boxtimes H) \leq \chi(G)\cdot \chi(H)$, (see \cite{2,4}).
}\end{remark}

\subsection{Chromatic Core Subgraph of the Lexicographic Product of Graphs}

Recall that in the lexicographical product the vertices $(v_i,u_j)$ and $(v_k, u_m)$ are adjacent if and only if, $v_i~ adj~ v_k$ or; $v_i = v_k$ and $u_j~adj~ u_m$.

\begin{lemma}\label{Lem-3.4}
For any general graph product of two graphs $G$ and $H$, denoted by $G\bigodot H$, is the graph with vertex set $V(G)\times V(H)$ such that the adjacency condition $v_i\sim v_k$ and $u_j \neq u_m$ results in more or an equal number of edges than the adjacency condition, $v_i\sim v_k$ and $u_j\sim u_m$. 
\end{lemma}
\begin{proof}~
Because graphs $G$ and $H$ have no loops, $u_j~adj~u_m \Rightarrow u_j \neq u_m$ hence, the first adjacency condition is met if $v_i~adj~v_k$. Also, $u_j~notadj~u_m \Rightarrow u_j \neq u_m$  hence, the first adjacency condition is met if $v_i~adj~v_k$. Therefore, the adjacency condition $v_i~ adj~ v_k$ and $u_j \neq u_m$ results in more or, an equal number of edges than the adjacency condition, $v_i~adj~v_k$ and $u_j~adj~u_m$. 
\end{proof}

\begin{conjecture}\label{Conj-3.6}{\rm 
For graphs $G$ and $H$ with chromatic core subgraphs $G'$ and $H'$ respectively, $G' \circledast H'$ is a chromatic core subgraph of $G \circledast H$.
}\end{conjecture}

In the lexicographical product the adjacency condition $v_i\sim v_k$ can be written as, $v_i\sim v_k$ and $u_j = u_m$ or $v_i\sim v_k$ and $u_j \neq u_m$. Hence, by Lemma \ref{Lem-3.4}, the lexicographical product results in more or equal number of edges than the strong product adjacency condition, $v_i\sim v_k$ and $u_j\sim u_m$. Since, all other adjacency conditions between the strong and the lexicographical products are similar, it follows that $|E(G\boxtimes H)| \leq |E(G\circledast H)|$. Hence, by mathematical induction on Conjecture \ref{Conj-3.5}, a similar result stating that $G' \circledast H'$ is a chromatic core subgraph of $G \circledast H$, is valid.

\subsection{Chromatic Core Subgraphs of Complement Graphs}

\noi First, observe the following straight forward result.

\begin{proposition}
For any finite, undirected connected simple graph $G$ its complement graph $\overline{G}$ has a proper subgraph (or induced subgraph) $\overline{H}$ such that $\chi(\overline{H}) = \chi(\overline{G})$ if and only if $G$ is neither a complete graph nor a 5-cycle.
\end{proposition}

It is obvious that if a self-complementary graph $G$ has a chromatic core subgraph $G'$ then $\overline{G}$ has a chromatic core subgraph $\overline{G}'$ that is isomorphic to $\overline{G'}$. 

\begin{theorem} For a finite, undirected connected graph $G$ with a chromatic core subgraph $G'$, we have that, $\overline{G}$ has a chromatic core subgraph $\overline{G}'$ that is isomorphic to $\overline{G'}$ if and only if either $P_3$ or self-complementary.
\end{theorem}
\begin{proof}~
Since $\overline{\overline{G}} = G$ we only have to prove the result for $P_3$ and for a self-complementary graph $G$. The converse follows by implication.
	
\textit{Case-1:} Consider $P_3$. From the definition of the complement graph, it follows that $\overline{P_3}= P_2 \cup K_1$. Clearly, both $P_3, \overline{P_3}$ have a $P_2$ as chromatic core subgraph. Hence, the result for $P_3$.

\textit{Case-2:} Let $G$ be any self-complementary graph with a chromatic core subgraph $G'$. Therefore, $G'$ has smallest $si(G')$ for all induced subgraphs $H$ of $G$ for which $\chi(H) = \chi(G)$. In $\overline{G}$, an induced subgraph $\overline{G'}$ must exist such that $\chi(\overline{G'}) = \chi(G')$ and $si(\overline{G'}) = si(G')$  such that $si(\overline{G'})$ is smallest for all induced subgraphs $\overline{H}$ of $\overline{G}$ for which $\chi(\overline{H}) =\chi(\overline{G'})$. Else, an induced subgraph $H^*$ of $\overline{G'}$ exists that satisfy the set conditions. But, that is a contradiction because $\overline{H^*}$ exists in $G$ such that $si(\overline{H^*}) < si(G')$ and $\chi(\overline{H^*}) = \chi(G)$. Hence, the result for self-complementary graphs follows.
\end{proof}

\begin{corollary}
\begin{enumerate}\itemsep0mm
\item[(i)]~ For a weakly perfect graph $G$ the largest clique $\overline{H}$ of the complement graph $\overline{G}$ is a chromatic core subgraph of $\overline{G}$.\\
\item[(ii)]~ For a non-weakly perfect graph $G$ the only unique chromatic core subgraph of $\overline{G}$, is $\overline{G}$ itself.
\end{enumerate}
\end{corollary}

Recall that a graph can be coloured in accordance with the Rainbow Neighborhood Convention (see \cite{9}). Such colouring utilises the colours $\mathcal{C} =\{c_1,c_2,c_3,\ldots,c_\ell\}$, where $\ell=\chi(G)$, and always colour vertices with maximum vertices in $G$ have the colour $c_1$, and maximum vertices in the remaining uncoloured vertices have the colour $c_2$ and following in this way until possible number of vertices have the colour $c_\ell$. Such a colouring is called a\textit{ $\chi^-$-colouring} of a graph.

\begin{proposition}
For a weakly perfect graph $G$ that is coloured in accordance with the rainbow neighbourhood convention, there exists a largest independent set $X$ such that in $G$ such that $c(u) = c_1,$ for all $u \in X$ and $\overline {\langle X\rangle}$ is a chromatic core subgraph of the complement graph $\overline{G}$.
\end{proposition}
\begin{proof}~
Assume that all largest independent sets are such that each has a number of vertices say, $t \geq 1$ which are not coloured $c_1$. Then clearly the colouring is not in accordance with the rainbow neighbourhood convention because colouring the $t$ vertices, $c_1$, yields another minimum proper colouring. Hence at least one such independent set must exists.

Consider this independent set to be $X$. Clearly by the definition of the complement grap $\langle X\rangle$ is a maximum induced clique of $\overline{G}$. Since $\overline{G}$ is also weakly perfect it follows that $\overline {\langle X\rangle}$ is a chromatic core subgraph of the complement graph $\overline{G}$.
\end{proof}

\subsection{Chromatic Core Subgraphs of the Line Graph of Trees}

It is noted that the line graph of a null graph (edgeless graph) of order $n$ denoted, $\mathfrak{N}_{0,n}$, remains a null graph of same order. Hence, the chromatic core subgraph of both $\mathfrak{N}_{0,n}$ and $L(\mathfrak{N}_{0,n})$ is $K_1$. For a path $P_n,\ n \geq 3$ the line graph $L(P_n) = P_{n-1}$. Therefore, the chromatic core subgraph for both is $P_2$. For certain classes of graphs the result in respect of chromatic core subgraphs is straight forward. Note the notation, $\geq \ell^{'s}$, means: $\in \{\ell, \ell +1, \ell +2, \ldots \}$.

\begin{theorem}
For a tree $T$ of order $n \geq 2$, its line graph $L(T)$ has a maximum clique as a chromatic core subgraph.
\end{theorem}
\begin{proof}~
Clearly, the degree sequences of a trees are characterised as either $(1,1)$ or $(1^{'s}, 2^{'s})$ or $(1^{'s}, \geq 3^{'s})$ or $(1^{'s}, 2^{'s}, \geq 3^{'s})$. Clearly for the possibilities $(1,1)$ or $(1^{'s}, 2^{'s})$ the line graph has a maximum cliques of order $\leq 2$. The corresponding clique is also the smallest induced subgraph with chromatic number equal to $\chi(L(T)) = 1~ or~ 2$. Therefore, it is a chromatic core subgraph of $L(T)$.

For the options $(1^{'s}, \geq 3^{'s})$ or $(1^{'s}, 2^{'s}, \geq 3^{'s})$ and without loss of generality, let $t = max\{\geq 3^{'s}\}$. Clearly the line graph has a maximum clique of order, $t$. The corresponding clique is also the smallest induced subgraph with chromatic number equal to $\chi(L(T)) = t$. Therefore, it is a chromatic core subgraph of $L(T)$.
\end{proof}

\section{Conclusion}

The paper introduced the notion of a chromatic core subgraph in respect of proper vertex colouring. The colours were all minimum parameter colours which is then called a chromatic colouring or in some work it is called, \textit{regular colouring}. The field of research by generalising the notion to edge colouring and other derivative colourings such as $J$-colouring, local colouring, dynamic colouring, co-colouring, grundy colouring, harmonious colouring, complete colouring, exact colouring, star colouring and others, offers a wide scope. The main application lies in large destruction, if possible, of a network carrying a defined technology capacity in accordance to a chosen colouring protocol whilst surviving with an induced subnetwork with equivalent technology capacity.

\subsection{New Direction of Research}

For the results in respect of strong and lexicographic products an alternative method of proof through induction on the structor indices is being investigated. The condition of simplicity of a graph is relaxed. Non-simplicity could result in the strong and lexicographic products through the new method of proof. Let the graph $G$ have vertex set $V(G) = \{v_1,v_2,v_3,\dots,v_n\}$. For example, note that a pseudo 2-path $P^\rho_2$ with multiple edges say, $v_1v_2$,~($t$ times) or $v_1v_2 = e_i,\ 1 \leq i \leq t$, has $\chi(P^\rho_2) = 2 = \chi(P_2)$. Therefore, a chromatic core subgraph of $P^\rho_2$ is $v_1e_iv_2,\ i \in \{1,2,3,\dots ,t\}$. Added to the aforesaid, it is known that, $\chi(P_{n\geq 2} \boxtimes P_{m \geq 2}) = 4 = \chi(C_{2n} \boxtimes P_{2m})$ with $K_4$ a chromatic core subgraph and, $\chi(C_{2n+1} \boxtimes C_{2m+1}) = 5$ with $K_5$ a chromatic core subgraph. These facts form the basis of the new direction of research.

\section*{Acknowledgement} 

The authors express their sincere gratitude to the anonymous referee whose valuable comments and theoretical guidance made it possible to finalise this research paper.

}

\begin{thebibliography}{99}
	
\bibitem{1} P. Ballen, Perfect graphs and the perfect graph theorems, \url{http://www.cis.upenn.edu/~pballen/brinkmann.pdf}

\bibitem{2} C. Berge, (1973). \textit{Graphs and hypergraphs}, CRC Press, Boca Raton.

\bibitem{3} J.A. Bondy and  U.S.R. Murty, (2008). \textit{Graph theory}, Springer, New York. 

\bibitem{4} M. Borowiecki, (1972). On chromatic number of products of two graphs, \textit{Colloq. Math.}, \textbf{25}, 49--52.

\bibitem{5} R. Hammack, W. Imrich and S. Klav\v{z}ar, (2011). \textit{Handbook of products graphs}, CRC Press, Boca Raton.

\bibitem{6} F. Harary, (2001). \textit{Graph theory}, Narosa Publ. House, New Delhi.

\bibitem{6a} S. Hedetniemi, (1966). Homomorphisms of graphs and automata, \textit{Technical Report 03105-44-T}, University of Michigan, USA.

\bibitem{7} J. Kok, C. Susanth, S.J. Kalayathankal, (2015). A Study on Linear Jaco Graphs, \textit{J. Inform. Math. Sci.}, \textbf{7}(2), 69--80.

\bibitem{8} J. Kok and S. Naduvath, An Essay on Comp\^{o}nent\u{a} Analysis of Graphs, \textit{Preprint, arXiv: 1705.02097v1 math.GM}.

\bibitem{9} J. Kok, S. Naduvath and M.K. Jamil, (2018). Rainbow Neighbourhood Number of Graphs, \textit{Natl. Acad. Sci. Lett.}, to appear.

\bibitem{10} D.B. West, (2001). \textit{Introduction to graph theory}, Prentice-Hall of India, New Delhi.

\end{thebibliography}
\end{document}